\documentclass[12pt]{amsart}
\usepackage{setspace}

\usepackage{latexsym,amsmath,amssymb,amsfonts,amscd,graphics}
\usepackage{blindtext}
\usepackage[left=1in,right=1in,bottom=1.5in]{geometry}
\usepackage[skip=4pt plus1pt, indent=10pt]{parskip}
\usepackage[mathscr]{eucal}
\usepackage{mathrsfs}
\usepackage{tikz-cd}
\usepackage{hyperref, cleveref}
\usepackage{thmtools}
\usepackage{multirow}
\hypersetup{
	colorlinks=true,
	linkcolor=blue,
	filecolor=blue,      
	urlcolor=cyan,
	citecolor=magenta}

\numberwithin{equation}{section}
\theoremstyle{plain}
\newtheorem{theorem}[equation]{Theorem}

\newtheorem{lemma}[equation]{Lemma}
\newtheorem{proposition}[equation]{Proposition}
\newtheorem{corollary}[equation]{Corollary}

\theoremstyle{remark}
\newtheorem{remark}[equation]{Remark}

\theoremstyle{definition}
\newtheorem{definition}[equation]{Definition}

\newcommand{\bP}{\mathbb{P}}

\newcommand{\bQ}{\mathbb{Q}}
\newcommand{\bZ}{\mathbb{Z}}
\newcommand{\bF}{\mathbb{F}}
\newcommand{\bC}{\mathbb{C}}

\newcommand{\calC}{\mathcal{C}}

\newcommand{\calO}{\mathcal{O}}

\newcommand{\calJ}{\mathcal{J}}

\newcommand{\GL}{\mathrm{GL}}

\newcommand{\SL}{\mathrm{SL}}

\newcommand{\Sing}{\mathrm{Sing}}
\newcommand{\im}{\mathrm{Im}}

\newcommand{\rank}{\mathrm{rank}}

\newcommand{\git}{/\kern-0.2em/}

\newcommand{\defect}{\sigma(X)}

\newcommand{\et}{\textup{\'et}}
\newcommand{\Id}{\mathrm{Id}}

\author{Lisa Marquand}
\address{Lisa Marquand, Courant Institute,
  251 Mercer Street,
  New York, NY 10012, USA}
\email{lisa.marquand@nyu.edu}

\author{Sasha Viktorova}
\address{Sasha Viktorova, Mathematics Division, National Center for Theoretical Sciences, National Taiwan University, Taipei City 106, Taiwan}
    \email{sasha.viktorova@ncts.ntu.edu.tw}

\title[]{Cubic fourfolds containing highly singular hyperplane sections} 
\date{\today}

\begin{document}

\begin{abstract}
    We construct five irreducible divisors in the moduli space of complex cubic fourfolds parametrising smooth cubic fourfolds that contain highly singular hyperplane sections. We prove that each is not a Noether-Lefschetz (or Hassett) divisor utilising the computational method developed by Addington--Auel. 
\end{abstract}

\maketitle
\section{Introduction}
Divisors in the moduli space of smooth cubic fourfolds, denoted $\calC$, have received a lot of recent attention, due to their connection with rationality constructions. Indeed, the recent work of \cite{KKPY} implies that the locus of rational cubic fourfolds is contained in the union of the \textit{Hassett divisors}, a countable union of irreducible divisors that parametrise cubic fourfolds containing a surface non-homologous to a complete intersection. These divisors were introduced by Hassett in \cite{hassett}, and they parametrise so-called \textit{special} cubic fourfolds.
It is a complementary problem to find explicit equations for cubic fourfolds, away from these Hassett divisors, where the theory of Hodge atoms of \cite{KKPY} guarantees irrationality.

To our knowledge, there are few known non-Hassett divisors in the moduli space $\calC$. Ranestand--Voisin constructed three divisors $D_{V-ap},D_{rk3},D_{IR}$ parametrising cubic fourfolds apolar to certain surfaces, and a fourth divisor $D_{copl}$ parametrising coplanar cubics. They proved that $D_{V-ap}$ was non-Hassett by carefully analysing its singularities \cite{RV17}. On the other hand, Addington--Auel used computational techniques to prove that $D_{rk3}$ and $D_{IR}$ were non-Hassett as well \cite{AA}. Further, Costa--Harvey--Kedlaya used $p$-adic methods to show the same conclusion holds for $D_{copl}$ \cite{CHK}.  More recently, Sammarco used techniques similar to \cite{RV17} to identify another non-Hassett divisor \cite{sammarco25}. It is shown in \cite{RV17} that $D_{V-ap}$ is different from $D_{rk3}$. However, it is not known if the other non-Hassett divisors, including the ones that we define in this work, are different.

In this paper, we consider smooth cubic fourfolds that have a hyperplane section that is more singular than expected. We define $D_{E_6}, D_{D_6}\subset \calC$ to be the locus parametrising smooth cubic fourfolds containing a hyperplane section $Y:=X\cap H$ with an $E_6$ or $D_6$ singularity, respectively. We prove:

\begin{theorem}\label{Main theorem}
    The loci $D_{E_6}, D_{D_6}\subset \calC$ are irreducible divisors in $\calC.$ Further, both divisors are not Hassett divisors.
\end{theorem}
In fact, one can write down the general form of an equation for a cubic contained in either $D_{E_6}, D_{D_6}$ (see \Cref{prop: equations} and Equation \ref{eqn: fourfold}), and thus applying \cite{KKPY} we can write down examples of irrational cubic fourfolds.

If $X$ is a general cubic fourfold, then necessarily all hyperplane sections have total Tjurina number (or total Milnor number) $\leq 5$ \cite[Lemma 3.8]{LSV}. 
Thus it is natural to consider hyperplane sections with total Tjurina number 6. 
A hyperplane section with six nodes is such an example, however cubic fourfolds with such a hyperplane section are contained in either the Hassett divisor $\calC_{12}$ (containing a cubic scroll) by \cite{flops}, or $\calC_{8}$ (containing a plane) if the singularities are not in general position. 
The condition for a cubic threefold $Y$ (necessarily singular) to contain either a cubic scroll or a plane was further studied in \cite{MV25}, where we showed this condition is equivalent to the positivity of the \textit{defect}: $\sigma(Y):=b_2(Y)-b_4(Y)>0$. 
The defect $\sigma(Y)$ is a global invariant of the cubic threefold, which can be computed using both the position of the singularities and the local numerical invariants. 
In particular, in order to obtain divisors away from $\calC_8\cup \calC_{12},$ we restrict our attention to cubic fourfolds containing singular hyperplane sections with $\sigma(Y)=0.$

With that in mind, we can consider loci $D_K\subset \calC$ parametrising cubic fourfolds with a hyperplane section $Y$ with multiple singularities of type $K=\sum_{i=1}^rK_i$, with each $K_i$ of $ADE$ type, and total Tjurina number 6. 
One would expect to find more non-Hassett divisors this way, as long as $\sigma(Y)=0$. 
We prove this is the case:

\begin{theorem}\label{Main theorem2}
\begin{spacing}{1.25}
    The loci $D_{D_5+A_1}, D_{D_4+A_2}\subset \calC$ are irreducible divisors, and $D_{D_4+2A_1}$ has two irreducible divisoral components $D_{D_4+2A_1}^0, D_{D_4+2A_1}^1,$ indexed by the value of $\sigma(Y)$. Further, the divisors $D_{D_5+A_1}, D_{D_4+A_2}, D_{D_4+2A_1}^0$ are non-Hassett divisors, while $D_{D_4+2A_1}^1=\calC_8.$
\end{spacing}
\end{theorem}
\vspace{-5mm}
We will also consider the locus $D_{T_{333}}\subset \calC$ that parametrises smooth cubic fourfolds containing a hyperplane section $Y$ with a $T_{333}$ singularity. This is a unimodular singularity of corank $3$ and Tjurina number $8$. One then expects this locus to be codimension $2$ in moduli, which we verify:
\begin{proposition}\label{prop: corank 3}
    The locus $D_{T_{333}}\subset \calC$ is irreducible of codimension 2. Further, it is not contained in a Hassett-divisor. 
\end{proposition}

Our interest in singularities of hyperplane sections of a cubic fourfold $X\subset \bP^5$ also stems from the study of compactified intermediate Jacobian fibrations $\pi:\calJ_X\rightarrow (\bP^5)^*$, first constructed in \cite{LSV} for general $X$. These are examples of hyperk\"ahler manifolds of $OG10$ type, and for a general $H\in (\bP^5)^*$ the fiber of $\pi$ is the intermediate Jacobian of the cubic threefold $X\cap H$. Although existence of such a hyperk\"ahler compactification is guaranteed for \textit{any} cubic by results of Sacc\'a \cite{sac2021birational}, it is non constructive. Recently, \cite{DM26} extend the construction of \cite{LSV} to \textit{very good cubic fourfolds} - cubics not containing a plane, a cubic scroll, or admitting a hyperplane section with a corank 3 singularity. Since the locus of cubic fourfolds admitting a hyperplane section with a corank 3 singularity is the closure of $D_{T_{333}}$ (see \Cref{non-general}), our results imply the following:
\begin{corollary} \label{verygood}
    A cubic fourfold $X\in \calC\setminus (\calC_8\cup\calC_{12}\cup D_{T_{333}})$ is very good.
    
    In particular, for a cubic $X\in \calC\setminus (\calC_8\cup\calC_{12}\cup D_{T_{333}})$, the construction of \cite{DM26} produces a smooth projective $OG10$-type hyperk\"ahler $\calJ_X$ with a Lagrangian fibration $\pi: \calJ_X\rightarrow (\bP^5)^*$ whose fibers are compactified Prym varieties of reduced, irreducible planar curves.
\end{corollary}

\subsection*{Outline}
We recall necessary background on cubic fourfolds, and singular cubic threefolds in Section \ref{sec: prelims}. In Section \ref{sec: construction}, we prove that the divisors described above are indeed irreducible divisors in the moduli of smooth cubic fourfolds. We obtain results on the deformations of singular hyperplane sections which may be of independent interest (Proposition \ref{prop: injectivity} and Lemma \ref{lemma: T333}). We complete the proofs of Theorems \ref{Main theorem} and \ref{Main theorem2} in Section \ref{sec: non Hassett} by utilising the computational methods of Addington--Auel. To implement their method to prove our divisors are non Hassett, we use code developed in \cite{AKPW, F2code}.

\subsection*{Funding}
The first author was supported by NSF grant DMS-2503390.

\subsection*{Acknowledgments} We would like to thank both Asher Auel and Jack Petok for introducing us to the work \cite{AKPW} and for explaining how to use the accompanying code \cite{F2code}.
We thank Franco Giovenzana, Dmitry Kubrak, Radu Laza, Grisha Papayanov, Elena Sammarco and Tim Seynnaeve for helpful discussions.

\section{Preliminaries}\label{sec: prelims}

Throughout, $X\subset \bP^5$ denotes a smooth cubic fourfold defined over $\bC.$ We denote by $\calC$ the moduli space of smooth cubic fourfolds, constructed using GIT \cite{Laz09}. Recall the following definition \cite{hassett}:

\begin{definition}
    A cubic fourfold $X$ is \textit{special} if it admits an algebraic cycle of codimension $2$ not homologous to a complete intersection.
\end{definition}

Equivalently, a smooth cubic fourfold $X$ is not special if and only if the algebraic lattice $A(X):=H^{2,2}(X)\cap H^4(X,\bZ)$
is generated by the class $h^2$, where $h\in H^{1,1}(X)$ is the hyperplane class.

Our goal is to construct divisors in $\calC$ parametrising geometrically relevant but not special cubic fourfolds. We will do so by parametrising cubics containing a hyperplane section with specific singularities. For a general hyperplane section $Y:=X\cap H$ of a general cubic fourfold $X$, we have that $Y$ is a smooth cubic threefold.
As $H$ varies, the hyperplane section can admit finitely many isolated singularities. Possible combinations of isolated singularities on cubic threefolds were classified in \cite{viktorova2023classification}. We recall some singularity theory definitions and notation.

\begin{definition}\label{def: Tjurina}
    For a germ of an isolated hypersurface singularity $(V(g),0)\subset (\bC^n,0)$ defined by a polynomial equation $g=0$, the Jacobian ideal of $g$ is
    \[
    J_g=\langle \partial g/\partial x_1,\dots, \partial f/\partial x_n\rangle\subset \bC[[x_1,\dots x_n]].
    \]
    The \textit{Tjurina algebra} of $(V(g),0)$ is defined to be
    \[
    T_g=\bC[[x_1,\dots x_n]]/\langle g,J_g\rangle,
    \]
    and the \textit{Tjurina number} is the dimension of the underlying vector space of $T_g$
    \[
    \tau(g)=\dim T_g.
    \]
\end{definition}
The Tjurina algebra of a hypersurface singularity $(V(g),0)$ is the same as the space of the first order deformations $T^1_{(V(g),0)}$ of $(V(g),0)$, and its basis generates a miniversal deformation of $(V(g),0)$ with a linear base space (see e.g. \cite[Corollary 1.17]{GLS}). When we consider the germ of an isolated singularity $p\in \Sing(Y)$, we denote $Y_p:=(Y,p).$
We denote by 
$$\tau(Y):=\sum_{p\in \Sing(Y)} \tau(Y_p)$$
the \textit{total Tjurina number} of a cubic threefold $Y$, where $\tau(Y_p)$ is the Tjurina number of $Y_p$.

The possible singularities of hyperplane sections of a general cubic $X$ were studied in connection to intermediate Jacobian fibrations by \cite{LSV}. In particular, they prove the following:

\begin{proposition} \cite[Lemma 3.8]{LSV}
    Let $X\subset\bP^5$ be a general cubic fourfold, and let $Y$ be any hyperplane section of $X$. Then the natural morphism
    \[
        \varphi: H^0(Y, \calO(1))\rightarrow \prod_{p\in\Sing  Y} T^1_{Y_p}
    \]
    is surjective. In other words, the family of deformations of $Y$ in $X$ induces a versal deformation of the singularities of $Y$.
\end{proposition}
Notice that $\dim H^0(Y,\calO(1))=5$. It follows that for a hyperplane section $Y$ of a general cubic fourfold $X$, the total Tijurina number $\tau(Y)$ is less than or equal to 5. Conversely, it is shown in \cite{Kazarian}, that a general cubic fourfold contains hyperplane sections with any combination of singularities possible with total Tjurina number less than or equal to $5$. Thus, we will consider cubic fourfolds with a hyperplane section $Y$ with $\tau(Y)=6.$ By a dimension count, this locus is expected to be a divisor in $\calC$.

Let $Y\subset\bP^4$ be a cubic threefold with exactly $m$ singularities of types $K_1,\ldots,K_m$. We say that $Y$ has the combination of singularities $K=K_1+\ldots+K_m$. We make the following definition:

\begin{definition} 
    If $Y$ has the combination of singularities $K$, we say that $Y$ is a \textit{$K$-cubic threefold}. Further, if $Y$ is a hyperplane section of a smooth cubic fourfold, we say that $Y$ is a \textit{$K$-hyperplane section}.
\end{definition}

After a linear change of coordinates, we can assume that one of the singularities of $Y\subset \bP^4$ is at $p=[1:0:0:0:0]$. We choose $p$ to be the worst singularity of $Y$. Then $Y$ can be defined by an equation of the form
\[
    f:=x_0f_2(x_1,x_2,x_3,x_4)+f_3(x_1,x_2,x_3,x_4)=0,
\]
where $f_2$ and $f_3$ are homogeneous polynomials of degree $2$ and $3$, respectively. The geometry of $Y$ is closely related to that of the complete intersection curve $C=V(f_2,f_3)\subset\bP^3$.

\begin{definition}
    We define the \textit{corank} of $p$ to be equal to the corank of the quadratic form $f_2(x_1,x_2,x_3,x_4)$.
\end{definition}

In what follows, we will be concerned with certain $ADE$ singularities and the singularity $T_{333}$ (the latter also known as $\tilde{E}_6$ and $P_8$), which are defined in $\bC^4$ by the following local equations:
\[
\begin{aligned}
    A_n&: x_1^2+x_2^2+x_3^2+x_4^{n+1}=0,\; n\geq 1;\\
    D_n&: x_1^2+x_2^2+x_4(x_3^2+x_4^{n-2})=0,\; n\geq 4;\\
    E_6&: x_1^2+x_2^2+x_3^3+x_4^4=0;\\
    T_{333}&: x_1^2+x_2^3+x_3^3+x_4^3+ax_2x_3x_4=0,\; a^3+27\neq 0.
\end{aligned}
\]
The corank is equal to $0$ for $A_1$ singularities, to $1$ for $A_n$ with $n>1$, to $2$ for $D_n$ and $E_n$, and to $3$ for $T_{333}$. Note that $ADE$ and $T_{333}$ singularities are quasihomogeneous, thus $\tau(g)$ coincides with the Milnor number $\mu(g)$. We have $\tau(A_n)=n$, $\tau(D_n)=n$, $\tau(E_6)=6$ and $\tau(T_{333})=8$. The $ADE$ singularities are sometimes called 0-modal, whereas $T_{333}$ is a unimodal singularity (or modality 1). The modality of a singularity $K$ is equal to the dimension of the strata in the miniversal deformation space of constant Milnor number $\mu:=\mu(p)$ \cite{Gab1}. 

If a singular point $p\in Y$ is of corank $2$ or $3$, the quadric $f_2$ defines either two planes or a double plane in $\bP^3$, respectively, which makes the curve $C$ and thus $Y$ easier to understand in comparison to smaller corank cases. By applying the classification of singularities of cubic threefolds \cite{viktorova2023classification}, we obtain:

\begin{corollary}
    The possible combinations of singularities on a cubic threefold $Y$ with $\tau(Y)=6$ and containing a singularity of corank larger than $1$ are:
    $$E_6, D_6, D_5+A_1, D_4+A_2 \text{ and } D_4+2A_1.$$ 
\end{corollary}

Let $Y$ be a $K$-cubic threefold with $K\in \{E_6, D_6, D_5+A_1, D_4+A_2\}$ as above. By \cite[Propositions 3.11, 3.14]{viktorova2023classification}, the curve $C$, obtained by projection from the singularity with highest $\tau$ in the above combinations, is a union of two irreducible plane cubic curves.
This implies that the defect $\sigma(Y)$ is $0$ by \cite[Theorem 1.1]{MV25}.
In contrast, if $Y$ is a $(D_4+2A_1)$-cubic, the curve $C$ can have two or three components, thus we can have $\sigma(Y)=0$ or $\sigma(Y)=1$. 
 Recall that the \textit{defect} of a normal projective variety measures the failure of $\bQ$-factoriality - if $\sigma(Y)>0$ then $Y$ contains a divisor that is not Cartier. In particular, cubic threefolds with $\sigma(Y)>0$ contain a plane or a cubic scroll by \cite[Theorem 1.5]{MV25}. We obtain:

\begin{corollary}
    Let $X$ be a cubic fourfold not containing a plane or a cubic scroll, i.e., $X\notin\calC_8\cup \calC_{12}.$ Then every hyperplane section $Y:=X\cap H$ has $\sigma(Y)=0.$
\end{corollary}
Thus our divisors will parametrise cubic fourfolds with hyperplane sections $Y$ with $\sigma(Y)=0.$

It will be useful it what follows to have the general form of a $K$-cubic threefold $Y$. Using the projection method, one obtains:

\begin{proposition} \label{prop: equations}
    Let $Y\subset \bP^4$ be a $K$-cubic threefold with $K\in \{E_6, D_6, D_5+A_1, D_4+A_2, T_{333}\}.$ After a linear changes of coordinates, we have that $Y=V(f_K)$, where $f_K$ is given as:
    \begin{align*}
        f_{E_6}&:=x_0x_1x_2+x_1q_1(x_1,x_3,x_4)+x_2q_2(x_2,x_3,x_4)+x_1x_2l(x_1,\ldots,x_4)+x_3^3,\\
        f_{D_6}&:=x_0x_1x_2+x_1q_1(x_1,x_3,x_4)+x_2q_2(x_2,x_3)+x_2x_4h(x_2,x_3)+x_1x_2l(x_1,\ldots,x_4)+x_3^2x_4,\\
        f_{D_5+A_1}&:=x_0x_1x_2+x_1q_1(x_1,x_3,x_4)+x_2q_2(x_3,x_4)+x_1x_2l(x_1,\ldots,x_4)+x_3^2x_4,\\
        f_{D_4+A_2}&:=x_0x_1x_2+x_1q_1(x_1,x_3,x_4)+x_2h^2(x_3,x_4)+x_1x_2l(x_1,\ldots,x_4)+x_3^2x_4+x_4^3,\\
        f_{T_{333}}&:=x_0x_1^2+f_3(x_1,\ldots,x_4).
    \end{align*}

    A $(D_4+2A_1)$-cubic threefold $Y\subset\bP^4$ with $\sigma(Y)=0$ or $\sigma(Y)=1$ can similarly be defined by:
    \begin{align*}
    f_{D_4+2A_1}^0&:=x_0x_1x_2+x_1q_1(x_3,x_4)+x_2q_2(x_3,x_4)+x_1x_2l(x_1,\ldots,x_4)+x_3^2x_4+x_4^3,\;\text{or}\\        f_{D_4+2A_1}^1&:=x_0x_1x_2+x_1q_1(x_1,x_3,x_4)+x_2x_4h(x_3,x_4)+x_1x_2l(x_1,\ldots,x_4)+x_3^2x_4+x_4^3.
    \end{align*}
    Here the polynomials $h$ and $l$ are linear, $q_1$ and $q_2$ are homogeneous of degree $2$, $f_3$ is homogeneous of degree $3$.
\end{proposition}
\begin{proof}
    The statement follows from \cite[Propositions 2.2, 3.11, 3.14]{viktorova2023classification}.
\end{proof}
\begin{remark}\label{non-general}
    The equations above define $K$-cubic threefolds for general choice of parameters; non-general ones may define cubic threefolds with additional singularities or a worse singularity at $p$. For instance, any cubic threefold singularity of corank $3$ can be defined by an equation of the form $f_{T_{333}}$.
\end{remark}

\section{Construction of the divisors}\label{sec: construction}
Let us fix $\bP^5$ with coordinates $x_0,\dots x_5.$ Recall that the (coarse) moduli space of smooth cubic fourfolds is constructed as $\calC:=U//\SL(6),$ where $U\subset \bP(H^0(\bP^5,\calO(3)))$ is the Zariski open subset of smooth cubic forms. Thus $\calC$ is a 20 dimensional quasi-projective variety. For $K$ as in Proposition \ref{prop: equations}, we denote by $W_K\subset U$ the locus of cubic forms
\begin{equation}\label{eqn: fourfold}
        F_K:=x_5q(x_0,\ldots,x_5)+f_K.
    \end{equation}

Every cubic fourfold with a $K$-hyperplane section is thus isomorphic to one defined by a cubic form in $W_K$.
Let $D_{K}\subset\calC$ denote the image of the closure of the locus $W_K\subset U$.
In this section, we will prove the following result:

\begin{theorem}\label{theorem: divisors}
    The loci $D_{K}\subset\calC$ are irreducible for $K\in\{E_6,D_6,D_5+A_1,D_4+A_2,T_{333}\}$. The locus $D_{D_4+2A_1}$ has two components, denoted $D_{D_4+2A_1}^0$ and $D_{D_4+2A_1}^1$, distinguished by whether the defect $\sigma$ is $0$ or $1$.
    
    The loci $D_{E_6}, D_{D_6}, D_{D_5+A_1}, D_{D_4+A_2}, D_{D_4+2A_1}^0, D_{D_4+2A_1}^1$ are irreducible divisors in $\calC$, and $D_{T_{333}}$ has codimension $2$ in $\calC$.
\end{theorem}
In order to prove Theorem \ref{theorem: divisors}, we need to analyse the deformation theory of a $K$-hyperplane section inside the ambient cubic fourfold $X$. In Section \ref{subsec:def theory}, we will prove that a cubic fourfold $X$ admitting a $K$-hyperplane section for $K$ as above only admits finitely many of them. In Section \ref{subsec: proof thm div}, we use this result to compute the dimension of $W_K$, along with the dimension of the $SL(6)$ orbit of a general element of $W_K$, allowing us to prove Theorem \ref{theorem: divisors}.  

\subsection{Deformation theory of \texorpdfstring{$Y$}{Y}}\label{subsec:def theory}

Let $K\in\{E_6, D_6,D_5+A_1, D_4+A_2, D_4+2A_1,  T_{333}\}$, and let $X\subset \bP^5$ be a general cubic fourfold containing a $K$-hyperplane section $Y$. We can assume that $X:=V(F_K)$ for $F_K$ as in Equation \ref{eqn: fourfold}, and that $Y:=V(f_K, x_5)$. 
Let $W_K\subset \bP(H^0(\bP^5, \calO(3))$ be the locus parameterising cubic fourfolds with a $K$-hyperplane section as above. Note that for $K=D_4+2A_1,$ we see that $W_{D_4+2A_1}=W_{D_4+2A_1}^0\cup W_{D_4+2A_1}^1$, corresponding to the two equations $f_{D_4+2A_1}^0$, $f_{D_4+2A_1}^1.$ 

We start by analysing the deformations of $Y$ in $X$. 

\begin{proposition}\label{prop: injectivity}
    Let $X$ be a cubic fourfold given by a general element in $W_{K}$ for $K\in \{E_6, D_6,D_5+A_1, D_4+A_2, D_4+2A_1\}$ and let $Y$ be the $K$-hyperplane section $X\cap \{x_5=0\}$. 
    Then the map
    \[
        \varphi: H^0(Y, \calO(1))\rightarrow \bigoplus_{p\in \Sing(Y)} T^1_{Y_p}
    \]
    is injective.
\end{proposition} 

\begin{proof}
    Notice that in order to prove the claim for each $K\in \{E_6, D_6,D_5+A_1, D_4+A_2\}$ it is enough to check the statement for one example of $X\in W_K$, since injectivity is an open condition. In the case of $D_4+2A_1$ one needs to prove the claim for one example in $W_{D_4+2A_1}^0,$ and one in $ W_{D_4+2A_1}^1$.
    We construct examples of smooth cubic fourfolds $X$ with a $K$-hyperplane section using the forms given in Proposition \ref{prop: equations} via \verb|MAGMA| - we list all the examples  in Section \ref{sec: non Hassett}. 
    
    We first prove the claim for $K=E_6$.    
    Consider the cubic fourfold $X:=\{F=0\}\subset \bP^5$, where
    \begin{align*}
        F =x_5&(89x_0^2-93x_5^2-51x_5^2+37x_0x_4-86x_4x_5+87x_4^2-60x_0x_3\\
       & -7x_3x_5-100x_3x_4-100x_0x_2-18x_2x_5+31x_2x_4+52x_2x_3\\
       &+71x_2^2-32x_0x_1+86x_1x_5+x_1x_4-30x_1x_3-41x_1x_2+13x_1^2)+f,
    \end{align*}
    where 
    \begin{align*}
     f = x_0&x_1x_2 - 24x_1^3 - 73x_1^2x_2 + 44x_1^2x_3 + 9x_1^2x_4 + 95x_1x_2^2 + 
    80x_1x_2x_3 - 62x_1x_2x_4 - 47x_1x_3^2 \\
    &+ 12x_1x_3x_4 - 7x_1x_4^2 + 15x_2^3 +
    33x_2^2x_3 - 88x_2^2x_4 + 29x_2x_3^2 - 33x_2x_3x_4 - 49x_2x_4^2 + x_3^3.
    \end{align*}
    Consider the hyperplane section $Y:=X\cap \{x_5=0\}$.
    By construction, $X$ is in $W_{E_6}$ (compare $F$ to Equation \ref{eqn: fourfold} and the equations in Proposition \ref{prop: equations}) and it is verifiable via \verb|MAGMA| that $X$ is smooth and $Y$ has a unique singularity at $p=[1:0:0:0:0]$ of type $E_6$.
       
    We first consider the base vector $T^1_{Y_p}$ of a miniversal deformation space for the $E_6$ singularity $(Y,p).$
    In the  chart $x_0\neq 0$, $Y$ is given by $\{x_5=g=0\}$, where
    \begin{align*}
        g=&x_1x_2-24x_1^3 - 73x_1^2x_2 + 44x_1^2x_3 + 9x_1^2x_4 + 95x_1x_2^2 + 80x_1x_2x_3 - 
    62x_1x_2x_4 - 47x_1x_3^2 \\
    &+ 12x_1x_3x_4 - 7x_1x_4^2 + 15x_2^3 + 
    33x_2^2x_3 - 88x_2^2x_4 + 29x_2x_3^2 - 33x_2x_3x_4 - 49x_2x_4^2 + x_3^3.
    \end{align*}
    One computes the ideal $J_g=\langle g,\partial g/\partial x_1,\dots \partial g/\partial x_4\rangle$ and finds that $T^1_{Y_p}= \bC[[x_1,\ldots,x_4]]/J_g$ is isomorphic as a $\bC$-vector space to $\bC^6$ with basis
    \[
        \{1,x_4,x_4^2, x_4^3, x_3, x_3x_4\}.
    \]

    Now we will consider a first order deformation of the hyperplane section $Y=X\cap\{x_5=0\}$ given by $x_5=\varepsilon l(x_0,\ldots,x_4)$, where $l(x_0,\dots x_4)$ is a linear form. The corresponding deformation of the hyperplane section is found by substituting into $f$ and expanding. We find that the deformation is given by the equation $f_{\varepsilon}=0$, where
    \[
        f_{\varepsilon}=f+
        \varepsilon\partial_5F(x_0,\ldots,x_4,0)l(x_0,\ldots,x_4).
    \]

   We analyse the global-to-local map $\varphi:H^0(Y,\calO(1))\rightarrow T^1_{Y_p}$.     
    Setting $x_0=1$ and rescaling the variables, we see that the map $\varphi$ is given by the vector space map
    \[
       l(1,x_1,\dots, x_4) \rightarrow \partial_5F(1,x_1\dots,x_4,0)l(1,x_1,\dots x_4) \mod J_g.
    \]
   Computing the corresponding matrix for the above linear map, one finds that $\rank(\varphi)=5, $ and thus $\varphi$ is injective as claimed. 
    
    For $K=D_6$, one repeats the same analysis as above for the cubic fourfolds given by Equation \ref{ex: E6} and \ref{ex: T_333} respectively. The hyperplane section $Y=X\cap\{x_5=0\}$ is a cubic threefold with an isolated singularity of the correct type at $p=[1:0:0:0:0:0]$. The remaining steps are identical to the proof for $K=E_6$, and can again be verified via \verb|MAGMA|.

    When $K$ is a combination of singularities, the linear map $\varphi$ becomes:
    $$l\mapsto (l\cdot \partial_5F|_{x_5=0, p_1},\dots, l\cdot \partial_5F|_{x_5=0, p_r}) \in \bigoplus_{p\in \Sing(Y)} T^1_{Y_p}.$$ For each $p\in \Sing(Y)$, one computes the map $H^1(Y,\calO(1))\rightarrow T_{Y_p}^1$ as before, and concatenating these maps gives the corresponding matrix for $\varphi.$ We compute for explicit cubic fourfolds given in Equations \ref{ex: D5+A1}, \ref{ex: D4+A2}, \ref{ex: D4+2A1 no defect} and \ref{ex: D4+2A1 with defect}. Again, one finds that $\rank(\varphi)=5$ in each of the above cases.
\end{proof}

We immediately obtain the following corollary:

\begin{corollary} \label{E6 discrete fibers}
    A general smooth cubic fourfold containing a $K$-hyperplane section for $K\in \{E_6, D_6,D_5+A_1, D_4+A_2,D_4+2A_1\}$ has finitely many $K$-hyperplane sections.
\end{corollary}
\begin{proof}
    The vector space $H^0(Y, \calO(1))$ is the tangent space to the space of deformations of the $K$-hyperplane section $Y$ in $X$. An $ADE$ singularity $p$ has modality zero, meaning the stratum of $T^1_p$ with constant Milnor number $\mu(p)$ is of dimension $0$, i.e., just the origin. It follows that if $X$ contains a $1$-parameter family of $K$-hyperplane sections, the map $\varphi$ as above cannot be injective.
\end{proof}

One cannot directly apply the argument for a singularity of type $T_{333}$. Indeed, $T_{333}$ is a unimodal singularity, with modality 1. Thus we need to analyse the image of the map $\varphi$ of Proposition \ref{prop: injectivity} in order to reach the desired conclusion.

\begin{lemma}\label{lemma: T333}
    A general smooth cubic fourfold containing a $T_{333}$-hyperplane section has finitely many $T_{333}$-hyperplane sections.
\end{lemma}
\begin{proof}
    We perform the analysis of Proposition \ref{prop: injectivity} in detail for $K=T_{333}.$ Consider the cubic fourfold $X:=\{F=0\}\subset \bP^5$ where

    \begin{equation*}
\begin{split}
    F=  x_5(- 71x_1^2 &- 20x_1x_2  + 93x_1x_3  + 80x_1x_4 + 42x_1x_5
    + 76x_0x_1- 8x_2^2  \\
    &+ 47x_2x_3 - 62x_2x_4+ 46x_2x_5 - 25x_0x_2  - 6x_3^2 + 65x_3x_4 - 70x_3x_5  \\
    & - 55x_0x_3 - 75x_4^2+ 94x_4x_5 - 15x_0x_4 + 41x_5^2 
    + 44x_0x_5 + 25x_0^2)+ f,
\end{split}
\end{equation*}
and 
\begin{equation*}
    \begin{split}
        f= x_0x_1^2 &+ -37x_1^3 - 13x_1^2x_2 - 28x_1^2x_3 + 53x_1^2x_4 - 20x_1x_3^2 - 31x_1x_2^2 \\
        &+ 71x_1x_2x_3 - 81x_1x_2x_4- 22x_1x_3x_4- 96x_1x_4^2 + 42x_2^3 + 55x_2^2x_3  \\
        &+ 36x_2^2x_4+ 63x_2x_3^2- 88x_2x_3x_4+ 64x_2x_4^2- 44x_3^3 - 59x_3^2x_4+ 44x_3x_4^2 - 19x_4^3.
        \end{split}
\end{equation*}
By construction, $X$ is smooth and $Y=X\cap\{x_5=0\}$ has a unique singularity at $p=[1:0:0:0:0]$ of type $T_{333}$.
We consider the local chart $x_0\neq 0$; after scaling $Y$ has equation $g(x_1,\dots x_4)=f(1,x_1,\dots x_5)=0.$

One computes the Jacobian ideal and finds that $T^1_{Y_p}$ is isomorphic as a $\bC$-vector space to $\bC^8$ with basis:
$$\{1, x_2, x_3, x_4, x_2x_4, x_3x_4, x_4^2, x_4^3\}.$$
The global-to-local map $\varphi:H^0(Y,\calO(1))\rightarrow T^1_{Y_p}$ is given by a matrix as before - one computes that $\rank(\varphi)=5$ and this map is injective.

Let $t_1,\dots,t_8$ be coordinates for $T^1_{Y_p}$ with respect to the basis above. It remains to check that $U:=\im(\varphi)\subset T^1_{Y_p}$ does not contain the 1-dimensional modular direction, i.e., the stratum where the singularity $T_{333}$ deforms. First we claim that the modular direction corresponds to $t_1=t_2=\dots =t_7=0.$
The local equation $g(x_1,\dots x_4)$ is quasihomogenous of degree one with weights $wt(x_1)=1/2$, $wt(x_2)=wt(x_3)=wt(x_4)=1/3$, and $T^1_{Y_p}$ is naturally a graded algebra. The total space of the miniversal deformation is given as:

$$g_t := g + \underbrace{t_1}_{\text{weight 0}} + \underbrace{t_2 x_2 + t_3 x_3 + t_4 x_4}_{\text{weight } 1/3} + \underbrace{t_5 x_2 x_4 + t_6 x_3 x_4 + t_7 x_4^2}_{\text{weight }2/3} + \underbrace{t_8 x_4^3}_{\text{weight 1}} = 0.$$ We see that deformations of total weight 1 (i.e. $t_1=t_2=\dots t_7=0$) preserve the degree of $g$ ($g_t$ remains quasihomogenous of degree 1), and thus preserves the singularity type. This fails as soon as one of $t_i$, $i\neq 8$, becomes nonzero. It follows that the modular direction is as claimed.

To conclude the proof, we must verify that the vector $(0,\dots ,0,1)\in T^1_{Y_p},$ a generator for this modular direction, is not contained in $U$. This can be easily checked via \verb|MAGMA|. It follows that since the modular line is not contained in $U$, it must intersect $U$ in finitely many points.
\end{proof}

\subsection{Proof of Theorem \ref{theorem: divisors}}\label{subsec: proof thm div}

We are now ready to prove \Cref{theorem: divisors}. We assume the notations of this section.

\begin{proof}[Proof of \Cref{theorem: divisors}]
    Let  $K\in\{E_6, D_6, D_5+A_1, D_4+A_2, T_{333}\}$ and consider the subset $W_K\subset \bP(H^0(\bP^5, \calO(3)))$ parametrising cubic fourfolds with equation of the form:
    \begin{equation}
        F_K:=x_5q(x_0,\ldots,x_5)+f_K=0
    \end{equation}
    as before.
    Notice that $W_K$ can be identified with its preimage in $H^0(\bP^5, \calO(3))$, as the Equation \ref{eqn: fourfold} contains an $x_0x_1x_2$ or $x_0x_1^2$ term with fixed coefficient $1$ for each $K$. 
    We abuse notation and call this subset also $W_K$. 
    By definition, the locus $D_{K}$ is the closure of the image in $\calC$ of an open subset of $W_K$. 
    Since $W_K$ is isomorphic to an affine space and thus is irreducible, $D_{K}$ is irreducible as well.

    \begin{spacing}{1.25}
    Similarly, let $W_{D_4+2A_1}^0,W_{D_4+2A_1}^1\subset H^0(\bP^5, \calO(3))$ be the loci parametrising cubic fourfolds given by$F_K=0$ with $f_K:=f^0_{D_4+2A_1},f^1_{D_4+2A_1}$ respectively, as in Proposition \ref{prop: equations}. 
    By the same argument each $W_{D_4+2A_1}^i$ can be identified with its image in $\bP(H^0(\bP^5,\calO(3)))$. 
    Each $W_{D_4+2A_1}^i$ is isomorphic to an affine space and is irreducible, thus each component $D_{D_4+2A_1}^0, D_{D_4+2A_1}^1$ is irreducible as well.
  \end{spacing}
  
    The quadratic form $q(x_0,\dots x_5)$ has $21$ parameters. The equation $f_K$ depends on $p_K$ parameters, where $p_K$ is as in the table below.

\begin{table}[h!]
\centering
\begin{tabular}{c|c|c|c|c|c|c|c}

$K$ & $E_6$ & $D_6$ & $D_5+A_1$& $D_4+A_2$ & $D_4+2A_1$, $\sigma=0$ & $D_4+2A_1$, $\sigma=1$ & $T_{333}$ \\ \hline
$p_K$   & 16   & 15   & 13   & 12  & 10  & 12 & 20 \\ 
\end{tabular}

\label{tab: parameter counts}
\end{table}
        
    Let $d_K$ be the dimension of the intersection of $W_K$ and the $\GL(6)$-orbit of a general element $w_K\in W_K$. 
    We have 
    \begin{equation}\label{eqn: dimension}
        \dim D_K=\dim W_K-\dim(W_K\cap (w\cdot\GL(6)))=(21+p_K)-d_K.
    \end{equation}
    We claim that $d_K$ is as in the following table:

    \begin{table}[h!]
\centering
\begin{tabular}{c|c|c|c|c|c|c|c}

$K$ & $E_6$ & $D_6$ & $D_5+A_1$& $D_4+A_2$ & $D_4+2A_1$, $\sigma=0$ & $D_4+2A_1$, $\sigma=1$ & $T_{333}$ \\ \hline
$d_K$   & 18   & 17   &  15  & 14   & 12   & 14 & 23\\ 
\end{tabular}

\label{tab: dK counts}
\end{table}
    
   Assuming the claim, the Theorem follows from Equation \ref{eqn: dimension}; it remains to compute $d_K$.
    
    \medskip
    \begin{spacing}{1.25}
    Consider the subgroup $G<\GL(6)$ consisting of elements preserving the plane $x_5=0$. By \Cref{E6 discrete fibers}, we have $\dim(W_K\cap w_K\cdot G)=\dim(W_K\cap w_K\cdot\GL(6))=d_K$.
    Indeed, let $Y_1,\ldots,Y_m$ be all the $K$-hyperplane sections of $X=V(w_K)\subset \bP^5$. There are suitable coordinate changes for each $i=1,\ldots,m$ such that $X$ is isomorphic to $X_i:=V(w_K^i)$ for $w_K^i\in W_K$ and $Y_i$ is given by $x_5=0$. Then  $W_K\cap w_K\cdot GL(6)=\cup_{i=1}^m(W_K\cap w_K^i\cdot G)$ and $\dim(W_K\cap w_K\cdot GL(6))=\dim(W_K\cap w_K^i\cdot G)$ for all $i=1,\ldots,m$ by the generality assumption.
    Notice that a general element $w_K\in W_K$ is smooth and thus GIT-semistable. Hence it has finite stabilizer and $d_K$ is equal to the dimension of $H_{w_K}:=\{g\in G \mid g\cdot w_K\in W_K\}$.
    \end{spacing}
    
    The group $G$ is isomorphic to the semidirect product $\GL(5)\ltimes(\bC^5\times \bC^*)$, where $\GL(5)$ is the subgroup of $\GL(6)$ acting on the first five coordinate functions $x_0,\ldots,x_4$. 
    The group $\bC^5\times \bC^*$ acts as the group of transformations of the form $x_i\mapsto \alpha_ix_5+x_i$, $x_5\mapsto \lambda x_5$, where $(\alpha_0,\ldots,\alpha_4,\lambda)\in \bC^5\times \bC^*$.
    
    We will now fix a general element $u_K$ in the set $U_K\subset H^0(\bP^4,\calO(3))$ of polynomials of the form $f_K\in \bC[x_0,\dots x_4]$ and describe the subset $H_{u_K}$ of $\GL(5)$ consisting of elements $h\in\GL(5)$ such that $u_K\cdot h$ is contained in $U_K$.  We will describe $H_{u_K}$ for each $K$ separately.  
    Notice that the unique singular point $p:=[1:0:0:0:0]$ of the threefold defined by $u_K$ has to be fixed. If $p$ is fixed, then an element in $H_{u_K}$ has the form: 
    $$x_0\mapsto a_{00}x_0+\ldots a_{04}x_4, \hspace{1cm} x_i\mapsto a_{i1}x_1+\ldots a_{i4}x_4, \text{ for } i=1,\dots 4.$$ For each $K$, we have $d_K=\dim(W_K\cap w_K\cdot G)=\dim(U_K\cap u_K\cdot\GL(5))+\dim(\bC^5\times \bC^*)=\dim H_{u_K}+6$.
    
      \smallskip
      
    \noindent\textbf{Case 1: }Let $K=E_6$. Recall that $u_{E_6}$ has the form:
    \[
        x_0x_1x_2+x_1q_1(x_1,x_3,x_4)+ x_2q_2(x_2,x_3,x_4)+x_1x_2l(x_1,\ldots,x_4)+x_3^3.
    \]
    As there is exactly one term in $u_{E_6}$ containing $x_0$, the subset $H_{u_K}\subset \GL(5)$ can only contain elements mapping $x_1\mapsto a_{11}x_1$ and $x_2\mapsto a_{22}x_2$, or $x_1\mapsto a_{12}x_2$ and $x_2\mapsto a_{21}x_1$, i.e., the variables $x_1$ and $x_2$ can only be rescaled and switched. 
    Moreover, since the term $x_0x_1x_2$ has a fixed coefficient $1$, it follows that $a_{00}a_{11}a_{22}=1$ or $a_{00}a_{12}a_{21}=1$. 
    Also, we note that $a_{34}=0$ (since we cannot have a nontrivial coefficient if front of $x_4^3$) and $a_{33}$ is a cube root of unity. 
    It is easy to see that for a transformation $h$ satisfying these properties $u_{E_6}\cdot h\in U_{E_6}$, thus we get that $H_{u_{E_6}}$ is generated by matrices    
    $$
    \begin{pmatrix}
        a_{00} & a_{01} & a_{02} & a_{03} & a_{04}\\
             0 & a_{11} &      0 &      0 &      0\\
             0 &      0 & a_{22} &      0 &      0\\
             0 & a_{31} & a_{32} & a_{33} &      0\\
             0 & a_{41} & a_{42} & a_{43} & a_{44}
    \end{pmatrix} \text{ and }
    \begin{pmatrix}
             1 &      0 &      0 &      0 &      0\\
             0 &      0 &      1 &      0 &      0\\
             0 &      1 &      0 &      0 &      0\\
             0 &      0 &      0 &      1 &      0\\
             0 &      0 &      0 &      0 &      1
    \end{pmatrix}, \text{ where } a_{00}a_{11}a_{22}=1 \text{ and } a_{33}^3=1.
    $$
    The set $H_{u_{E_6}}$ is of dimension $12$. We see $d_{E_6}=\dim H_{u_{E_6}}+6=12+6=18$.

    \medskip    

    \noindent \textbf{Case 2:} In the $K=D_6$ case, we proceed similarly. Again, there is exactly one term in $u_{D_6}$ containing $x_0$, and we can only have $x_1\mapsto a_{11}x_1$ and $x_2\mapsto a_{22}x_2$, or $x_1\mapsto a_{12}x_2$ and $x_2\mapsto a_{21}x_1$; we also have $a_{00}a_{11}a_{22}=1$ or $a_{00}a_{12}a_{21}=1$. Since $x_3^2x_4$ is the only monomial of $u_{D_6}$ in variables $x_3$ and $x_4$, we conclude that $a_{34}=a_{43}=0$ and $a_{33}^2a_{44}=1$. Finally, we notice that the form of $u_{D_6}$ is not symmetric in $x_1$ and $x_2$, thus transformations with $x_1\mapsto a_{12}x_2$ and $x_2\mapsto a_{21}x_1$ cannot be in $H_{u_{D_6}}$. We see that $H_{u_{D_6}}$ consists of matrices
    $$
    \begin{pmatrix}
        a_{00} & a_{01} & a_{02} & a_{03} & a_{04}\\
             0 & a_{11} &      0 &      0 &      0\\
             0 &      0 & a_{22} &      0 &      0\\
             0 & a_{31} & a_{32} & a_{33} &      0\\
             0 & a_{41} & a_{42} &      0 & a_{44}
    \end{pmatrix}, \text{ where } a_{00}a_{11}a_{22}=1 \text{ and } a_{33}^2a_{44}=1.
    $$
    The set $H_{D_6}$ is of dimension $11$ and $d_{D_6}=11+6=17$.

    \medskip

    \noindent\textbf{Case 3:} In the $K=D_5+A_1$ case, we see that $H_{u_{D_5+A_1}}$ consists of matrices:
$$
\begin{pmatrix}
        a_{00} & a_{01} & a_{02} & a_{03} & a_{04}\\
        0 & a_{11} &      0 &      0 &      0\\
        0 &      0 & a_{22} &      0 &      0\\
        0 & a_{31} &      0 & a_{33} &      0\\
        0 & a_{41} &      0 &      0 & a_{44}
\end{pmatrix}, 
$$
where $a_{00}a_{11}a_{22}=1$ and $a_{33}^2a_{44}=1.$
The set of matrices is $9$-dimensional and $d_K=15$. 

\medskip

\noindent\textbf{Case 4:} In the case $K=D_4+A_2$, we see that $H_{u_{D_4+A_2}}$ consists of matrices:
$$
\begin{pmatrix}
        a_{00} & a_{01} & a_{02} & a_{03} & a_{04}\\
        0 & a_{11} &      0 &      0 &      0\\
        0 &      0 & a_{22} &      0 &      0\\
        0 & a_{31} &      0 & a_{33} & a_{34}\\
        0 & a_{41} &      0 & a_{43} & a_{44}
\end{pmatrix},$$
where $a_{00}a_{11}a_{22}=1$ and there are 18 solutions for $a_{33},a_{34},a_{43},a_{44}$. The set of matrices is $8$-dimensional and $d_K=14$. 

\medskip

\noindent\textbf{Case 5:} In the case $K=D_4+2A_1$ and $\sigma=0$, we see that $H_{u_{D_4+2A_1}}^0$ consists of matrices:
$$
\begin{pmatrix}
        a_{00} & a_{01} & a_{02} & a_{03} & a_{04}\\
        0 & a_{11} &      0 &      0 &      0\\
        0 &      0 & a_{22} &      0 &      0\\
        0 &      0 &      0 & a_{33} & a_{34}\\
        0 &      0 &      0 & a_{43} & a_{44}
\end{pmatrix},
\begin{pmatrix}
        1 &      0 &      0 &      0 &      0\\
        0 &      0 &      1 &      0 &      0\\
        0 &      1 &      0 &      0 &      0\\
        0 &      0 &      0 &      1 &      0\\
        0 &      0 &      0 &      0 &      1
\end{pmatrix},
$$
where $a_{00}a_{11}a_{22}=1$ and there are 18 solutions for $a_{33},a_{34},a_{43},a_{44}$.
The set of matrices is $6$-dimensional and $d_K=12$.
\medskip

\noindent \textbf{Case: 6} In the case $K=D_4+2A_1$ and $\sigma=1$, we see that $H_{u_{D_4+2A_1}}^1$ consists of matrices:
$$
\begin{pmatrix}
        a_{00} & a_{01} & a_{02} & a_{03} & a_{04}\\
        0 & a_{11} &      0 &      0 &      0\\
       0 &      0 & a_{22} &      0 &      0\\
       0 & a_{31} &      0 & a_{33} & a_{34}\\
      0 & a_{41} &      0 & a_{43} & a_{44}
\end{pmatrix}, $$
where $a_{00}a_{11}a_{22}=1$ and there are 18 solutions for $a_{33},a_{34},a_{43},a_{44}$. The set of matrices is $8$-dimensional and $d_K=14$. 
    
    \noindent \textbf{Case 7:} In the $K= T_{333}$ case, we can only have $x_1\mapsto \pm x_1$ and we see that $H_{u_{T_{333}}}$ consists of matrices
    $$
    \begin{pmatrix}
        a_{00} & a_{01} & a_{02} & a_{03} & a_{04}\\
             0 &  \pm 1 &      0 &      0 &      0\\
             0 & a_{21} & a_{22} & a_{23} & a_{24}\\
             0 & a_{31} & a_{32} & a_{33} & a_{34}\\
             0 & a_{41} & a_{42} & a_{43} & a_{44}
    \end{pmatrix},
    $$
    thus $d_{T_{333}}=17+6=23$, proving the claim.
\end{proof}

\begin{remark}
    A cubic fourfold with a $(D_4+2A_1)$-section with $\defect=1$, is necessarily special by \cite[Theorem 1.5]{MV25} - such a cubic fourfold contains a plane. An example of such a cubic is given by:
    \begin{equation}\label{ex: D4+2A1 with defect}
\begin{split}
    F= x_5(&- 42x_1^2+ 93x_1x_2+ 29x_1x_3 - 75x_1x_4 - 16x_1x_5 + 80x_1x_0 - 68x_2^2- 4x_2x_3\\
    &+ 67x_2x_4 - 21x_2x_5 + 28x_2x_0 - 
    97x_3^2 + 29x_3x_4 - 98x_3x_5 + 38x_3x_0- 10x_4^2 \\
    &+ 25x_4x_5 + 2x_4x_0 - 97x_52 - 67x_5x_0 + 30x_0^2)+x_0x_1x_2+x_1(13x_1^3- 35x_1^2x_3 \\
    &+ 79x_1^2x_4- 8x_1x_3^2 + 30x_1x_3x_4 
    - 16x_1x_4^2)\\
    &+x_2x_4(- 74x_3  - 53x_4)+x_1x_2(51x_1- 25x_2- 86x_3 - 88x_4)   + x_3^2x_4  + x_4^3. 
\end{split}
\end{equation}
Since a general cubic fourfold with a $(D_4+2A_1)$-section admits finitely many sections, it follows from Corollary \ref{E6 discrete fibers} and Theorem \ref{theorem: divisors} that $D_{D_4+2A_1}^1$ coincides with $\calC_{8}.$
\end{remark}

\section{Proof of main theorem}\label{sec: non Hassett}

In this section, we will prove the following result, which implies Theorems \ref{Main theorem}, \ref{Main theorem2}:
\begin{theorem}\label{theorem: main}
    The irreducible divisors $D_{E_6}, D_{D_6}, D_{D_5+A_1},D_{D_4+A_2},D_{D_4+2A_1}^0\subset \calC$ are not Hassett divisors. Further, the codimension 2 locus $D_{T_{333}}\subset \calC$ is not contained in a Hassett divisor.
\end{theorem}

We will prove \Cref{theorem: main} with computer assisted methods, following the techniques of Addington and Auel outlined in \cite{AA}. Their method is based on a strategy of van Luijk \cite{vanLuijik} (see also \cite{EJ1, EJ2}), for producing explicit K3 surfaces. We outline their method in Section \ref{subsec: AA method}, before applying to our divisors in Section \ref{section proof: main}. 

\subsection{The Addington-Auel method}\label{subsec: AA method}

In \cite{AA}, Addington and Auel propose a computer-assisted method of verifying whether a divisor $D$ in the moduli space of cubic fourfolds $\calC$ is not a Hassett divisor. The starting point of their method is the proposition below, which we have written to apply directly to the cubic fourfold case:

\begin{proposition}\label{prop: char poly} \cite[Proposition 2.1]{AA}
    Let $R=\bZ_{(2)}$, $k=\bF_2$, $L=\bQ$, and let $X_R$ be a smooth cubic fourfold over $R$. 
    Let $X^{an}$ denote the complex manifold associated to the complex cubic fourfold $X_{\bC}$. 
    Let $\Phi: X_k\rightarrow X_k$ be the absolute Frobenius morphism, and let $l\neq 2$ be a prime number. 
    Denote by
    $$V=H^4_{\et,prim}(X_{\bar{k}},\bQ_l(2))\cong\bQ_l^{22},$$ and let $\Phi^*$ be the automorphism of $V$ induced by $\Phi\times 1$ on $X_k\times \bar{k}$. 
    
    Then the rank of $H^{2,2}_{prim}(X^{an},\bZ)$ is bounded above by the number of eigenvalues of $\Phi^*$ that are roots of unity, counted with multiplicity.
\end{proposition}

\begin{corollary} \cite[Section 2]{AA} \label{non-special}
    If none of the eigenvalues of $\Phi^*$ are roots of unity, or equivalently if the characteristic polynomial $\chi(t):=\det(t\cdot\Id_V-\Phi^*)$ has no cyclotomic factor, then $X^{an}$ is a non-special cubic fourfold. In particular, if $\chi(t)$ is irreducible over $\bQ$ and not all of its coefficients are integers, then $X^{an}$ is non-special.
\end{corollary}

Addington and Auel developed and implemented an algorithm \cite[Section 3]{AA} for computing the characteristic polynomial as above.
We briefly outline the steps from \cite{AA} one can follow to confirm that a divisor $D\subset\calC$ is non-Hassett:
\begin{itemize}
    \item[(1)] Choose a smooth cubic fourfold $X\in D$ defined by an equation with rational coefficients.
    \item[(2)] Check if $X$ has good reduction modulo $2$.
    \item[(3)] If yes in (2), count the points of $X_{\bF_{2^m}}$ for all $1\leq m\leq 11$ using the conic bundle algorithm explained in \cite[Section 3]{AA}.
    \item[(4)] Find the characteristic polynomial $\chi(t)$ using the point counts from the previous step (using the Lefschetz trace formula and Newton's identities, see \cite[Section 2]{AA}).
    \item[(5)] Check if $\chi(t)$ is irreducible over $\bQ$ and not all of its coefficients are integers.
    \item[(6)] If yes in (5), conclude that $X$ is non-special by Corollary \ref{non-special} and that $D$ is non-Hassett.
\end{itemize}
The computations in steps (3), (4) and (5) can be easily done thanks to the library \verb|CubicLib.m| introduced in the accompanying code \cite{F2code} to \cite{AKPW}. Alternatively, if one has the full census of cubic fourfolds over $\bF_2$ already compiled, one can simply look up the characteristic polynomial for the reduction $X_{\bF_2}.$

\subsection{Proof of Theorem \ref{theorem: main}} \label{section proof: main}

\begin{proof}[Proof of \Cref{theorem: main}]\label{proof: main theorem}
    
    For each $K\in \{E_6, D_6,D_5+A_1,D_4+A_2,D_4+2A_1,T_{333}\}$, we randomly generate an equation of the form of Equation \ref{eqn: fourfold} via \verb|Magma| and verify that the cubic $X$ given by $F=0$ is smooth over $\bC$. 
    The $K$-hyperplane section is obtained as $Y:=\{x_5=0\}\cap X,$ so $X$ defines a point in $D_K$ by construction. 
    We verify via \verb|Magma| that $Y$ has exactly the expected singularity combination $X$. 
    The polynomial $F_2=0$ is the reduction of $X$ modulo $2$; one verifies this is smooth over $\overline{\bF}_2.$ 
    One then computes the characteristic polynomial $\chi(t)$. 
    We repeat these steps until we find a cubic fourfold as above with $\chi(t)$ irreducible over $\bQ$. 
    This can take many attempts - we present a successful output below. 
    
\noindent\textbf{Case 1: $E_6$.}
 We considered the cubic with defining equation:
\begin{equation}\label{ex: E6}
\begin{split}
F =x_5&(89x_0^2-93x_5^2-51x_5^2+37x_0x_4-86x_4x_5+87x_4^2-60x_0x_3\\
       & -7x_3x_5-100x_3x_4-100x_0x_2-18x_2x_5+31x_2x_4+52x_2x_3\\
       &+71x_2^2-32x_0x_1+86x_1x_5+x_1x_4-30x_1x_3-41x_1x_2+13x_1^2)+x_0x_1x_2\\
       & - 24x_1^3 - 73x_1^2x_2 + 44x_1^2x_3 + 9x_1^2x_4 + 95x_1x_2^2 + 
    80x_1x_2x_3 - 62x_1x_2x_4 - 47x_1x_3^2 \\
    &+ 12x_1x_3x_4 - 7x_1x_4^2 + 15x_2^3 +
    33x_2^2x_3 - 88x_2^2x_4 + 29x_2x_3^2 - 33x_2x_3x_4 - 49x_2x_4^2 + x_3^3.
\end{split}
\end{equation}
Over $\bF_2:$
\begin{align*}
F_2:= x_1^2x_2 &+ x_1^2x_4 + x_1^2x_5 + x_1x_2^2 + x_1x_2x_5 + x_1x_2x_0 + x_1x_3^2 + x_1x_4^2 + x_1x_4x_5 + x_2^3 \\
&+ x_2^2x_3 + x_2^2x_5 + x_2x_3^2 + x_2x_3x_4 + x_2x_4^2 + x_2x_4x_5 + x_3^3 + x_3x_5^2 + x_4^2x_5 \\
&+ x_4x_5x_0 + x_5^3 + x_5^2x_0 + x_5x_0^2.
\end{align*}
The characteristic polynomial is:
\[
    \chi(t)=t^{22}-\frac{1}{2}t^{19}+\frac{1}{2}t^{16}+\frac{1}{2}t^{15}-\frac{1}{2}t^{14}-\frac{1}{2}t^{12}+\frac{1}{2}t^{11}-\frac{1}{2}t^{10}-\frac{1}{2}t^8+\frac{1}{2}t^7+\frac{1}{2}t^6-\frac{1}{2}t^3+1.
\]

\noindent\textbf{Case 2: $D_6$.} We considered the cubic with defining equation:

\begin{equation}\label{ex: D6} 
\begin{split}   F =  x_5&(93x_5^2 -39x_0^2- 83x_5x_0+ 36x_0x_4 + 72x_4x_5- 13x_0x_3 - 78x_4^2 + 67x_3x_4 \\
    &- 17x_3x_5 + 17x_3^2- 25x_2x_3 + 38x_2x_4 - 98x_2x_5 + 3x_0x_2+ 94x_2^2+ 66x_1x_5 \\
    &+ 66x_0x_1- 70x_1x_4- 24x_1x_3- 40x_1x_2- 90x_1^2)+ x_0x_1x_2+ x_1(58x_1^2+ 8x_1x_3 \\
    &+ 60x_1x_4+ 17x_3^2 + x_3x_4  - 7x_4^2)+x_2(- 77x_2^2 + 29x_2x_3- 75x_3^2)+x_2x_4(88x_2+85x_3)\\
    &+x_1x_2(- 95x_1 - 69x_3 - 82x_4)+ x_3^2x_4 .
    \end{split}
\end{equation}
Over $\bF_2:$
\begin{align*}
    F_2 = x_1^2x_2 &+ x_1x_2x_3 + x_1x_2x_0 + x_1x_3^2 + x_1x_3x_4 + x_1x_4^2 + x_2^3 + x_2^2x_3 + x_2x_3^2 + x_2x_3x_4 + x_2x_3x_5 \\
    &+ x_2x_5x_0 + x_3^2x_4 + x_3^2x_5 + x_3x_4x_5 + x_3x_5^2 + x_3x_5x_0 + x_5^3 + x_5^2x_0 + x_5x_0^2.
\end{align*}
The characteristic polynomial is:
\begin{align*} 
\chi(t) = t^{22}& + t^{21} + 1/2t^{20} - 1/2t^{19} - t^{18} - 3/2t^{17} - 1/2t^{16} - 1/2t^{15} + 1/2t^{14} + 1/2t^{13} + t^{12} \\&+ 1/2t^{11}
+ t^{10} + 1/2t^{9} 
+ 1/2t^8 - 1/2t^7 - 1/2t^6 - 3/2t^5 - t^4 - 1/2t^3 + 1/2t^2 + t + 1.
\end{align*}

\noindent\textbf{Case 3: $D_5+A_1$.} We considered the cubic with defining equation:
\begin{equation}\label{ex: D5+A1}
\begin{split}
    F= x_5&(34x_1^2- 100x_1x_2- 62x_1x_3+ 44x_1x_4 + 4x_1x_5 - 80x_0x_1 + 17x_2^2+ 45x_2x_3- 11x_2x_4  \\
    &- 29x_2x_5- 13x_0x_2+ 17x_3^2 - 65x_3x_4 + 64x_3x_5 + 38x_0x_3 + 34x_4^2 - 47x_4x_5 - 74x_4x_0 \\
    &+ 8x_5^2 - 99x_0x_5 - 91x_0^2)+x_0x_1x_2+x_1(-37x_1^2- 31x_1x_3 + 88x_1x_4+ 11x_3^2 - 78x_3x_4 \\
    & + 58x_4^2)+x_2(- 17x_3^2 - 79x_3x_4  - 15x_4^2)+x_1x_2(+ 40x_1- 5x_2 + 100x_3 + 71x_4)  + x_3^2x_4. 
\end{split}
\end{equation}
Over $\bF_2:$
\begin{align*}
    F_2 = &x_1^3 + x_1^2x_3 + x_1x_2^2 + x_1x_2x_4 + x_1x_2x_0 + x_1x_3^2 + x_2^2x_5 + x_2x_3^2 + x_2x_3x_4 + x_2x_3x_5  \\
    &+ x_2x_4^2+ x_2x_4x_5 + x_2x_5^2 + x_2x_5x_0 + x_3^2x_4 + x_3^2x_5 + x_3x_4x_5 + x_4x_5^2 + x_5^2x_0 + x_5x_0^2.
\end{align*}
The characteristic polynomial is:
\begin{align*} 
\chi(t) = t^{22} - 1/2t^{21} &+ 3/2t^{20} - 1/2t^{19} + 1/2t^{18} - 1/2t^{16} - t^{13} + t^{12} \\
&- 3/2t^{11} + t^{10} - t^9 - 1/2t^6 + 1/2t^4 - 1/2t^3 + 3/2t^2 - 1/2t + 1.
\end{align*}

\noindent\textbf{Case 4: $D_4+A_2$.} We considered the cubic with defining equation:
\begin{equation}\label{ex: D4+A2}
\begin{split}
    F= x_5(77x_1^2&+ 13x_1x_2- 95x_1x_3+ 96x_1x_4 - 29x_1x_5 + x_1x_0 - 64x_2^2+ x_2x_3 \\
    &+ 72x_2x_4 - 40x_2x_5 - 4x_2x_0+ 84x_3^2 - 76x_3x_4 + 92x_3x_5 + 91x_3x_0\\
    &- 61x_4^2 + 12x_4x_5 + 64x_4x_0 - 21x_5^2 - 80x_5x_0 + 95x_0^2)+x_0x_1x_2\\
    &+x_1(-96x_1^2+ 29x_1x_3 + 36x_1x_4+82x_3^2 + 21x_3x_4  - 4x_4^2)+x_2(23x_3+4x_4)^2\\
    &+x_1x_2(80x_1+ 63x_2+ 32x_3 + 94x_4)+ x_3^2x_4  + x_4^3. \end{split}
\end{equation}

The characteristic polynomial for $F$ considered over $F_2$ is:
\begin{align*} 
\chi(t) = t^{22} + t^{21} &- t^{20} - 3/2t^{19} + 1/2t^{18} + 1/2t^{17} - t^{16} + 3/2t^{14} + 1/2t^{13} 1/2t^{12}\\
&- 1/2t^{11} - 1/2t^{10} + 1/2t^9 + 3/2t^8 - t^6 + 1/2t^5 + 1/2t^4 - 3/2t^3 - t^2 + t + 1.
\end{align*}

\noindent\textbf{Case 5: $D_4+2A_1$, $\sigma=0$.} We considered the cubic with defining equation:
\begin{equation}\label{ex: D4+2A1 no defect}
\begin{split}
    F= x_5(&-31x_1^2+ 19x_1x_2- 97x_1x_3- 29x_1x_4 - 73x_1x_5 - 53x_1x_0 + 33x_2^2\\
    &- 66x_2x_3- 57x_2x_4 - 29x_2x_5 + 9x_2x_0+ 36x_3^2 - 47x_3x_4 - 3x_3x_5 + 83x_3x_0\\
    &+ 90x_4^2 - 48x_4x_5 - 37x_4x_0 + 60x_5^2 + 31x_5x_0 + 73x_0^2)+ + x_0x_1x_2+x_1(- 29x_3^2 \\
    &- 98x_3x_4  - 34x_4^2)+x_2(93x_3^2 - 72x_3x_4  - 26x_4^2)+x_1x_2(-77x_1  - 76x_2+ 72x_3 - 71x_4)\\
    &+ x_3^2x_4  + x_4^3.
\end{split}
\end{equation}

The characteristic polynomial for $F$ considered over $F_2$ is:
\begin{align*} 
\chi(t) = t^{22} - t^{21} &+ 1/2t^{20} - t^{18} + 1/2t^{17} + 1/2t^{16} + 1/2t^{14} + 1/2t^{13} - t^{12} \\
&+ 1/2t^{11} - t^{10} + 1/2t^9 + 1/2t^8 + 1/2t^6 + 1/2t^5 - t^4 + 1/2t^2 - t + 1.
\end{align*}

\noindent\textbf{Case 7: $T_{333}$.} We considered the cubic with defining equation:

\begin{equation}\label{ex: T_333}
\begin{split}
    F= -37&x_1^3 - 13x_1^2x_2 - 28x_1^2x_3 + 53x_1^2x_4 - 71x_1^2x_5 + x_1^2x_0 - 31x_1x_2^2 + 71x_1x_2x_3 - 81x_1x_2x_4 \\
    &- 20x_1x_2x_5 - 20x_1x_3^2 - 22x_1x_3x_4 + 93x_1x_3x_5 - 96x_1x_4^2 + 80x_1x_4x_5 + 42x_1x_5^2 \\
    &+ 76x_1x_5x_0 + 42x_2^3 + 55x_2^2x_3 + 36x_2^2x_4 - 8x_2^2x_5 + 63x_2x_3^2 - 88x_2x_3x_4 + 47x_2x_3x_5 \\
    &+ 64x_2x_4^2 - 62x_2x_4x_5 + 46x_2x_5^2 - 25x_2x_5x_0 - 44x_3^3 - 59x_3^2x_4 - 6x_3^2x_5 + 44x_3x_4^2 \\
    &+ 65x_3x_4x_5 - 70x_3x_5^2 - 55x_3x_5x_0 - 19x_4^3 - 75x_4^2x_5 + 94x_4x_5^2 - 15x_4x_5x_0 + 41x_5^3 \\
    &+ 44x_5^2x_0 + 25x_5x_0^2.
\end{split}
\end{equation}
Over $\bF_2:$
\begin{align*}
    F_2= x_1^3& + x_1^2x_2 + x_1^2x_4 + x_1^2x_5 + x_1^2x_0 + x_1x_2^2 + x_1x_2x_3 + x_1x_2x_4 + x_1x_3x_5 + x_2^2x_3 + x_2x_3^2 \\
    &+ x_2x_3x_5 + x_2x_5x_0 + x_3^2x_4 + x_3x_4x_5 + x_3x_5x_0 + x_4^3 + x_4^2x_5 + x_4x_5x_0 + x_5^3 + x_5x_0^2.
\end{align*}
The characteristic polynomial is:
\begin{align*}
    \chi(t)= t^{22} - t^{20} + 1/2t^{18} + 1/2t^{11} + 1/2t^4 - t^2 + 1.
\end{align*}
\end{proof}

\bibliographystyle{alpha}
\bibliography{bibliography}
\end{document}